\documentclass{article}
\usepackage{amssymb}
\usepackage{amsmath}

\newtheorem{theorem}{Theorem}

\newtheorem{corollary}[theorem]{Corollary}

\newtheorem{definition}[theorem]{Definition}

\newenvironment{proof}[1][Proof]{\noindent\textbf{#1.} }{\ \rule{0.5em}{0.5em}}
\input{tcilatex}
\begin{document}

\title{On Pfaffian and determinant of one type of skew centrosymmetric
matrices }
\author{{\small Fatih YILMAZ}$^{a}${\small \thanks{%
Corresponding author} Tomohiro SOGABE}$^{b}${\small \thanks{%
e-mails: fatihyilmaz@gazi.edu.tr, sogabe@na.cse.nagoya-u.ac.jp,
e.kirklar@gazi.edu.tr },} {\small Emrullah KIRKLAR}$^{a}${\small \ } \\
%EndAName
$^{a}${\small Polatl\i\ Art and Science Faculty, Gazi University, Turkey}\\
$^{b}${\small Graduate School of Engineering, Nagoya University, Japan}}
\maketitle

\begin{abstract}
This paper is dedicated to compute Pfaffian and determinant of one type of
skew centrosymmetric matrices in terms of general number sequence of second
order.

\textbf{Key words:} Pfaffian; determinant; skew centrosymmetric matrix
\end{abstract}

\section{Introduction}

The \textit{determinant\ }is one of the basic parameters in matrix theory.
For an $n$-square matrix $A$, it is defined by 
\begin{equation*}
\det (A)=\underset{\sigma \in S_{n}}{\sum }\mathrm{sgn}(\sigma )\underset{i=1%
}{\overset{n}{\prod }}a_{i\sigma (i)}.
\end{equation*}

The \textit{Pfaffian }which is\textit{\ }intimately related to determinant,
was introduced by Cayley, denoted by $Pf(A),$ is defined by 
\begin{equation*}
\mathrm{Pf}(A)=\underset{\pi \in S_{n}}{\sum }\mathrm{sgn}(\pi )_{2n}%
\underset{i=1}{\overset{n}{\prod }}a_{\pi (2i-1)\pi (2i)}.
\end{equation*}

The Pfaffian of a skew symmetric matrix is a quantity closely related to the
determinant. Cayley's Theorem says that the square of the Pfaffian of a
matrix is equal to the determinant of the matrix. In other words, for an $n$%
-square skew-symmetric matrix $A$, 
\begin{equation*}
\det (A)=[\mathrm{Pf}(A)]^{2}.
\end{equation*}

A matrix $A$ is a \textit{centrosymmetric matrix} if $A=JAJ^{-1}$ where $J$
is anti-diagonal matrix whose anti-diagonal entries are one and others are
zero. If $A=-AJA^{-1},$ it is said to be \textit{skew-centrosymmetric matrix}%
. This matrix family have wide applications in many fields of science such
as, numerical solution of certain differential equations, digital signal
processing, information theory, statistics, linear systems theory, some
Markov processes and so on (see \cite[2, 3, 4, 5, 6]{1}).

This paper highlights the close connections among the Pfaffian, the
determinant and general number sequences of second order.

\section{Pfaffian and Skew centrosymmetric matrices}

\begin{definition}
Let us define $n$-square matrices $A_{n}=[a_{i,j}]$ and $B_{n}=[b_{i,j}]$ in
the given form 
\begin{equation*}
\lbrack a_{i,j}]=\left\{ 
\begin{array}{rl}
a, & \text{for }j=i+1 \\ 
-a, & \text{for }i=j+1 \\ 
0, & \text{otherwise}%
\end{array}%
\right. ,\text{ \ \ \ \ \ }[b_{i,j}]=\left\{ 
\begin{array}{rl}
(-1)^{i+1}b, & \text{for }i+j=n+1 \\ 
0, & \text{otherwise}%
\end{array}%
\right.
\end{equation*}%
where $1\leq i,j\leq n.$
\end{definition}

\begin{definition}
Let us define $2\times 2$-block matrices as below.%
\begin{equation*}
\mathcal{F}_{n}=\left( 
\begin{array}{cc}
A_{k} & B_{k} \\ 
(-1)^{k}B_{k} & A_{k}%
\end{array}%
\right) \text{ and \ \ }\mathcal{G}_{n}=\left( 
\begin{array}{cc}
A_{k} & -B_{k} \\ 
(-1)^{k+1}B_{k} & A_{k}%
\end{array}%
\right) .
\end{equation*}
\end{definition}

For example, for $n=10$, the $n$-square ($n=2k,$ $k$ is any positive
integer) skew centrosymmetric matrix $\mathcal{F}_{n}$ will be in the
following form:

\begin{equation*}
\mathcal{F}_{10}=\left( 
\begin{array}{ccccc|ccccc}
0 & a & 0 & 0 & 0 & 0 & 0 & 0 & 0 & b \\ 
-a & 0 & a & 0 & 0 & 0 & 0 & 0 & -b & 0 \\ 
0 & -a & 0 & a & 0 & 0 & 0 & {b} & 0 & 0 \\ 
0 & 0 & -a & 0 & a & 0 & -b & 0 & 0 & 0 \\ 
0 & 0 & 0 & -a & 0 & b & 0 & 0 & 0 & 0 \\ \hline
0 & 0 & 0 & 0 & -b & 0 & a & 0 & 0 & 0 \\ 
0 & 0 & 0 & b & 0 & -a & 0 & a & 0 & 0 \\ 
0 & 0 & {-b} & 0 & 0 & 0 & -a & 0 & a & 0 \\ 
0 & b & 0 & 0 & 0 & 0 & 0 & -a & 0 & a \\ 
-b & 0 & 0 & 0 & 0 & 0 & 0 & 0 & -a & 0%
\end{array}%
\right) .
\end{equation*}

\begin{definition}
Let us consider couple-recurrences given below.%
\begin{align*}
&f_{n}=bg_{n-1}+a^{2}f_{n-2}\text{ \ \ for \ }f_{1}=b, \\
&g_{n}=-bf_{n-1}+a^{2}g_{n-2}\text{ \ \ for \ }g_{1}=-b.
\end{align*}
\end{definition}

Then we have the following theorem:

\begin{theorem}
For $n=2k,$%
\begin{equation*}
f_{k}=\mathrm{Pf}(\mathcal{F}_{n})\text{ \ \ and \ \ }g_{k}=\mathrm{Pf}(%
\mathcal{G}_{n}),
\end{equation*}%
where $f_{-1}=0,$ $f_{0}=1$ and $g_{-1}=0,$ $g_{0}=1.$
\end{theorem}

\begin{proof}
Let us prove the theorem by using mathematical induction method. For $k=1$, 
\begin{equation*}
\mathcal{F}_{2}=\left( 
\begin{array}{cc}
A_{1} & B_{1} \\ 
-B_{1} & A_{1}%
\end{array}%
\right) =\left( 
\begin{array}{cc}
0 & b \\ 
-b & 0%
\end{array}%
\right) \text{ and }\mathcal{G}_{2}=\left( 
\begin{array}{cc}
A_{1} & -B_{1} \\ 
B_{1} & A_{1}%
\end{array}%
\right) =\left( 
\begin{array}{cc}
0 & -b \\ 
b & 0%
\end{array}%
\right).
\end{equation*}%
In this case, 
\begin{equation*}
f_{1}=\mathrm{Pf}(\mathcal{F}_{2})=b,\text{ \ \ \ \ \ \ \ }g_{1}=\mathrm{Pf}(%
\mathcal{G}_{2})=-b.
\end{equation*}%
Assume that the recurrences hold for all $t\leq k.$ Then they hold for $%
k=t+1 $.%
\begin{equation}
\mathcal{F}_{2t+2}=\left( 
\begin{array}{c|c}
A_{t+1} & B_{t+1} \\ \hline
(-1)^{t+1}B_{t+1} & A_{t+1}%
\end{array}%
\right) =\left( 
\begin{tabular}{c|cccc|c}
$0$ & $a$ & $0$ & $\cdots $ & $0$ & $b$ \\ \hline
$-a$ &  &  &  &  & $0$ \\ 
$0$ &  & $A_{t}$ & $-B_{t}$ &  & $\vdots $ \\ 
$\vdots $ &  & $(-1)^{t+1}B_{t}$ & $A_{t}$ &  & $0$ \\ 
$0$ &  &  &  &  & $a$ \\ \hline
$-b$ & $0$ & $\cdots $ & $0$ & $-a$ & $0$%
\end{tabular}%
\right).  \label{10}
\end{equation}%
From the expansion formula along with $2t+2$ column of (\ref{10}), it
follows that 
\begin{equation}
\mathrm{Pf}(\mathcal{F}_{2t+2})=b\mathrm{Pf}(\mathcal{G}_{2t})+a\mathrm{Pf}(%
\mathcal{M}_{2t})=bg_{t}+a\mathrm{Pf}(\mathcal{M}_{2t}),  \label{20}
\end{equation}%
where%
\begin{equation}
\mathcal{M}_{2t}=\left( 
\begin{tabular}{cc|cccc}
$0$ & $a$ & $0$ & $\cdots $ & $\cdots $ & $0$ \\ 
$-a$ & $0$ & $a$ & $0$ & $\cdots $ & $0$ \\ \hline
$0$ & $a$ &  &  &  &  \\ 
$\vdots $ & $0$ &  & $A_{t-1}$ & $B_{t-1}$ &  \\ 
$\vdots $ & $\vdots $ &  & $(-1)^{t-1}B_{t-1}$ & $A_{t-1}$ &  \\ 
$0$ & $0$ &  &  &  & 
\end{tabular}%
\right)  \label{30}
\end{equation}%
From the expansion formula along with 1st row of (\ref{30}), it follows that 
\begin{equation}
\mathrm{Pf}(\mathcal{M}_{2t})=a\mathrm{Pf}(\mathcal{F}_{2t-2})=af_{t-1}.
\label{40}
\end{equation}%
From (\ref{20}) and (\ref{40}), we have%
\begin{equation*}
f_{t+1}=bg_{t}+a^{2}f_{t-1}.
\end{equation*}%
The recurrence for $g_{t+1}$ can be obtained in a similar manner.
\end{proof}

\begin{corollary}
$f_{n}=(-1)^{n-1}bf_{n-1}+a^{2}f_{n-2}$ with $f_{-1}=0$\ and $f_{1}=1.$
\end{corollary}

\section{Determinant of the skew centrosymmetric matrix}

In this section, we consider determinant of the matrix $\mathcal{F}_{n}$ ($%
n=2k$). It is a well-known fact \cite{14} for block matrices:%
\begin{equation*}
\left\vert 
\begin{array}{cc}
A & B \\ 
C & D%
\end{array}%
\right\vert =\det (AD-BC)
\end{equation*}%
if it verifies $AC=CA$. Taking into account this property, determinant of
the matrix $\mathcal{F}_{n}$ is

\begin{equation*}
\left\vert \mathcal{F}_{n}\right\vert =\left\vert \mathcal{T}_{k}\right\vert
=\det \left( 
\begin{array}{ccccc}
-a^{2}+b^{2} & 0 & a^{2} &  &  \\ 
0 & -2a^{2}+b^{2} & 0 & \ddots &  \\ 
a^{2} & 0 & \ddots & \ddots & a^{2} \\ 
& \ddots & \ddots & -2a^{2}+b^{2} & 0 \\ 
&  & a^{2} & 0 & -a^{2}+b^{2}%
\end{array}%
\right) _{k\times k}.
\end{equation*}

Sogabe and El-Mikkawy \cite{10} considered a fast block diagonalization of $%
k $-tridiagonal matrices using permutation matrices. Exploiting this method,
we can rearrange the matrix $\mathcal{T}_{k}$.

\textbf{(\textit{i}) For }$\mathbf{k}$\textbf{\ is odd}

For this let us define the following matrices: 
\begin{equation*}
H_{\frac{k-1}{2}}=\left\{ 
\begin{array}{l}
-2a^{2}+b^{2},\text{ for }i=j \\ 
a^{2},\text{ for }i=j+1\text{ and }j=i+1 \\ 
0,\text{ otherwise}%
\end{array}%
\right. \text{ }
\end{equation*}%
and%
\begin{equation*}
K_{\frac{k+1}{2}}=\left\{ 
\begin{array}{l}
-a^{2}+b^{2},\text{ }i=j=1\text{ and }i=j=\frac{k+1}{2} \\ 
-2a^{2}+b^{2},\text{ }i=j=2(1)\frac{k-1}{2} \\ 
a^{2},\text{ for }i=j+1\text{ and }j=i+1 \\ 
0,\text{ otherwise}%
\end{array}%
\right. .
\end{equation*}%
Then,

\begin{equation*}
P^{T}\mathcal{F}_{k}P=\left( 
\begin{array}{c|c}
H_{\frac{k-1}{2}} & 0 \\ \hline
0 & K_{\frac{k+1}{2}}%
\end{array}%
\right), \text{ }
\end{equation*}%
where permutation matrix $P$ is determined by using method in \cite{10}.
Obviously, 
\begin{equation*}
\det (P^{T}\mathcal{T}_{k}P)=\det \mathcal{T}_{k}=\det \mathcal{F}_{n}=\det
(H_{\frac{k-1}{2}})\det (K_{\frac{k+1}{2}}).
\end{equation*}

\textbf{(\textit{ii})} \textbf{For }$\mathbf{k}$\textbf{\ is even}

Let us define

\begin{equation*}
N_{\frac{k}{2}}=\left\{ 
\begin{array}{l}
-a^{2}+b^{2},\text{ }i=j=\frac{k}{2} \\ 
-2a^{2}+b^{2},\text{ }i=j=1(1)\frac{k}{2}-1 \\ 
a^{2},\text{ for }i=j+1\text{ and }j=i+1 \\ 
0,\text{ otherwise}%
\end{array}%
\right.
\end{equation*}%
and

\begin{equation*}
Q_{\frac{k}{2}}=\left\{ 
\begin{array}{l}
-a^{2}+b^{2},\text{ }i=j=1 \\ 
-2a^{2}+b^{2},\text{ }i=j=2(1)\frac{k}{2} \\ 
a^{2},\text{ for }i=j+1\text{ and }j=i+1 \\ 
0,\text{ otherwise}%
\end{array}%
\right. .
\end{equation*}%
Then,%
\begin{equation*}
P^{T}\mathcal{F}_{k}P=\left( 
\begin{array}{c|c}
N_{\frac{k}{2}} & 0 \\ \hline
0 & Q_{\frac{k}{2}}%
\end{array}%
\right) \text{ }.
\end{equation*}%
Obviously, 
\begin{equation*}
\det (P^{T}\mathcal{T}_{k}P)=\det \mathcal{T}_{k}=\det \mathcal{F}_{n}=\det
(N_{\frac{k}{2}})\det (Q_{\frac{k}{2}}).
\end{equation*}%
It can be seen that $\det (N_{\frac{k}{2}})=\det (Q_{\frac{k}{2}})$.

El-Mikkawy \cite{11} obtained determinant of tridiagonal matrix. That is,%
\begin{equation*}
v_{i}=\left\vert 
\begin{array}{ccccc}
d_{1} & a_{1} & 0 & \ldots & 0 \\ 
b_{2} & d_{2} & a_{2} & \ddots & \vdots \\ 
0 & b_{3} & d_{3} & \ddots & 0 \\ 
\vdots & \ddots & \ddots & \ddots & a_{i-1} \\ 
0 & \ldots & 0 & b_{i} & d_{i}%
\end{array}%
\right\vert,
\end{equation*}%
where $v_{i}=d_{i}v_{i-1}-b_{i}a_{i-1}v_{i-2}$ for $v_{0}=1$ and $v_{-1}=0$.
By exploiting \cite{11} and Laplace expansion:

For\ $k$ is even%
\begin{equation*}
\det (N_{\frac{k}{2}})=\det (Q_{\frac{k}{2}})=(-a^{2}+b^{2})w_{\frac{k}{2}%
-1}-a^{4}w_{\frac{k}{2}-2}.
\end{equation*}

For $k$ is odd%
\begin{align*}
&\det (K_{\frac{k+1}{2}})=\left( -a^{2}+b^{2}\right) ^{2}w_{\frac{k-3}{2}%
}-2a^{4}(-a^{2}+b^{2})w_{\frac{k-5}{2}}+a^{8}w_{\frac{k-7}{2}}, \\
&\det (H_{\frac{k-1}{2}})=w_{\frac{k-1}{2}},
\end{align*}%
where 
\begin{equation*}
w_{i}=\left\vert 
\begin{array}{cccc}
-2a^{2}+b^{2} & a^{2} & \ldots & 0 \\ 
a^{2} & -2a^{2}+b^{2} & \ddots & \vdots \\ 
\vdots & \ddots & \ddots & a^{2} \\ 
0 & \ldots & a^{2} & -2a^{2}+b^{2}%
\end{array}%
\right\vert.
\end{equation*}%
Here $w_{i}=(-2a^{2}+b^{2})w_{i-1}-a^{4}w_{i-2}$ for $w_{0}=1$ and $w_{-1}=0$%
.

Consequently, for $n=2k$,

\begin{enumerate}
\item[$i$)] If $k$ is odd, 
\begin{equation*}
\det \mathcal{F}_{n}=\det \mathcal{T}_{k}=w_{\frac{k-1}{2}}\left( \left(
-a^{2}+b^{2}\right) ^{2}w_{\frac{k-3}{2}}-2a^{4}(-a^{2}+b^{2})w_{\frac{k-5}{2%
}}+a^{8}w_{\frac{k-7}{2}}\right).
\end{equation*}

\item[$ii$)] If $k$ is even, $\det \mathcal{F}_{n}=\det \mathcal{T}%
_{k}=\left( (-a^{2}+b^{2})w_{\frac{k}{2}-1}-a^{4}w_{\frac{k}{2}-2}\right)
^{2}. $
\end{enumerate}

\section{Examples}

Let us consider the matrix $\mathcal{F}_{n}~$($n=2k$). Then examples of the
Pfaffians and the determinants are shown in Tables 1 and 2 respectively.
Here $F_{n},P_{n}$ and $J_{n}$ are $n$th Fibonacci, Pell and Jacobsthal
numbers, respectively.

\begin{table}[h]
\caption{Examples of the Pfaffians.}
\label{tbl:Pfaffian}\centering
%\begin{equation*}
\begin{tabular}{c|l|l|l}
& $a=i,b=1$ & \ $a=i,b=2$ & $a=i\sqrt{2},b=1$ \\ \cline{2-4}
$k$ & \hphantom{-}\hphantom{-} $Pf{\small (}\mathcal{F}_{2k}{\small )}$ & %
\hphantom{-}\hphantom{-} $Pf{\small (}\mathcal{F}_{2k}{\small )}$ & %
\hphantom{-}\hphantom{-} $Pf{\small (}\mathcal{F}_{2k}{\small )}$ \\ 
\hline\hline
\multicolumn{1}{l|}{$\ \ \ \ \ \ \ \ 1$} & $\hphantom{-} F_{2}=1$ & $%
\hphantom{-} P_{2}=2$ & $\hphantom{-} J_{2}=1 $ \\ 
\multicolumn{1}{l|}{$\ \ \ \ \ \ \ \ 2$} & $-F_{3}=-2$ & $-P_{3}=-5$ & $%
-J_{3}=-3$ \\ 
\multicolumn{1}{l|}{$\ \ \ \ \ \ \ \ 3$} & $-F_{4}=-3$ & $-P_{4}=-12$ & $%
-J_{4}=-5$ \\ 
$4$ & $\hphantom{-} F_{5}=5$ & $\hphantom{-} P_{5}=29$ & $\hphantom{-}
J_{5}=11$ \\ 
$5$ & $\hphantom{-} F_{6}=8$ & $\hphantom{-} P_{6}=70$ & $\hphantom{-}
J_{6}=21$ \\ 
$6$ & $-F_{7}=-13$ & $-P_{7}=-169$ & $-J_{7}=-43$ \\ 
$7$ & $-F_{8}=-21$ & $-P_{8}=-408$ & $-J_{8}=-85$ \\ 
$8$ & $\hphantom{-} F_{9}=34$ & $\hphantom{-} P_{9}=985$ & $\hphantom{-}
J_{9}=171$ \\ 
$\vdots $ & \hphantom{-}\hphantom{-}\hphantom{-}\hphantom{-} $\vdots $ & %
\hphantom{-}\hphantom{-}\hphantom{-}\hphantom{-} $\vdots $ & \hphantom{-}%
\hphantom{-}\hphantom{-}\hphantom{-} $\vdots $ \\ \hline
$\equiv 0,1(\mathrm{mod}4)$ & $\hphantom{-} F_{k+1}$ & $\hphantom{-} P_{k+1}$
& $\hphantom{-} J_{k+1}$ \\ 
$\equiv 2,3(\mathrm{mod}4)$ & $-F_{k+1}$ & $-P_{k+1}$ & $-J_{k+1}$%
\end{tabular}%
%
%
%
%\end{equation*}
\end{table}
\begin{table}[h]
\caption{Examples of the determinants.}
\label{tbl:determinant}\centering
\begin{tabular}{c|c|c|c}
& $a=i,b=1$ & $a=i,b=2$ & $a=i\sqrt{2},b=1$ \\ \cline{2-4}
$k$ & $\det {\small (}\mathcal{F}_{2k}{\small )}$ & $\det {\small (}\mathcal{%
F}_{2k}{\small )}$ & $\det {\small (}\mathcal{F}_{2k}{\small )}$ \\ 
\hline\hline
\multicolumn{1}{c|}{${\small 1}$} & ${\small F}_{2}^{2}$ & ${\small P}%
_{2}^{2}$ & ${\small J}_{2}^{2}$ \\ 
\multicolumn{1}{c|}{${\small 2}$} & ${\small F}_{3}^{2}$ & ${\small P}%
_{3}^{2}$ & ${\small J}_{3}^{2}$ \\ 
\multicolumn{1}{c|}{{\small \ }$\ 3\ \ $} & ${\small F}_{4}^{2}$ & ${\small P%
}_{4}^{2}$ & ${\small J}_{4}^{2}$ \\ 
${\small 4}$ & ${\small F}_{5}^{2}$ & ${\small P}_{5}^{2}$ & ${\small J}%
_{5}^{2}$ \\ 
${\small 5}$ & ${\small F}_{6}^{2}$ & ${\small P}_{6}^{2}$ & ${\small J}%
_{6}^{2}$ \\ 
${\small 6}$ & ${\small F}_{7}^{2}$ & ${\small P}_{7}^{2}$ & ${\small J}%
_{7}^{2}$ \\ 
${\small 7}$ & ${\small F}_{8}^{2}$ & ${\small P}_{8}^{2}$ & ${\small J}%
_{8}^{2}$ \\ 
${\small 8}$ & ${\small F}_{9}^{2}$ & ${\small P}_{9}^{2}$ & ${\small J}%
_{9}^{2}$ \\ 
${\small \vdots }$ & ${\small \vdots }$ & ${\small \vdots }$ & ${\small %
\vdots }$ \\ 
${\small t}$ & ${\small F}_{t+1}^{2}$ & ${\small P}_{t+1}^{2}$ & ${\small J}%
_{t+1}^{2}$%
\end{tabular}%
\end{table}


\begin{thebibliography}{9}
\bibitem{1} A. L. Andrew, Centrosymmetric matrices, SIAM Rev., 40 (1988),
697--698,

\bibitem{2} A. L. Andrew, Eigenvectors of certain matrices, Linear Alg.
Appl., 7 (1973), 151--162.

\bibitem{14} F. Zhang, Matrix Theory Basic Results and Techniques, Springer,
1999.

\bibitem{4} L. Datta and S. Morgera, On the reducibility of centrosymmetric
matrices- applications in engineering problems, Circuits Systems Signal
Process, 8 (1989), 71--96.

\bibitem{5} M. El-Mikkawy and F. Atlan, On solving centrosymmetric linear
systems, Applied Mathematics, 4 (2013), 21--32.

\bibitem{6} J. Weaver, Centrosymmetric (cross-symmetric) matrices, their
basic properties, eigenvalues, and eigenvectors, Amer. Math. Monthly, 92
(1985), 711--717.

\bibitem{7} R. Vein and P. Dale, Determinants and Their Applications in
Mathematical Physics, Springer-Verlag New York, Inc, 1999.

\bibitem{10} T. Sogabe and M. E. Mikkawy, Fast block diagonalization of $k$%
-tridiagonal matrices, Appl. Math. Compute., 218 (2011), 2740--2743.

\bibitem{11} M. El-Mikkawy, A note on a three-term recurrence for a
tridiagonal matrix, Appl. Math. Compute., 139 (2003), 503--511.
\end{thebibliography}
\end{document}